\documentclass[reqno,11pt]{amsart}
\usepackage{amsmath} 
\usepackage{amssymb} 
\usepackage{amsfonts} 
\usepackage{physics} 
\usepackage{esint} 
\usepackage{braket} 
\usepackage[toc,page]{appendix} 
\newtheorem{theorem}{Theorem}[section]
\newtheorem{proposition}[theorem]{Proposition}
\newtheorem{conjecture}{Conjecture} 

\theoremstyle{remark}
\newtheorem{remark}[theorem]{Remark}

\newcommand{\R}{\mathbb{R}}

\newcommand{\HH}{\mathcal{H}} 

\usepackage[breakable]{tcolorbox}
\usepackage{xcolor}
\usepackage{soulutf8}
\soulregister\cite7
\soulregister\underline7
\soulregister\eqref7
\soulregister\ref7
\definecolor{mycolor}{RGB}{229,228,200}
\definecolor{pinkcolor}{RGB}{247,202,193}

\newtcolorbox{warningbox}{
    colback=pinkcolor,    
    colframe=pinkcolor,   
    boxrule=0pt,
    sharp corners,      
    left=0pt,
    right=0pt,
    top=0pt,
    bottom=0pt,
    breakable}

\newtcolorbox{changebox}{
    colback=mycolor,    
    colframe=mycolor,   
    boxrule=0pt,
    sharp corners,      
    left=0pt,
    right=0pt,
    top=0pt,
    bottom=0pt,
    breakable}

\author[C. Labourie]{Camille Labourie}
\author[A. Lemenant]{Antoine Lemenant}

\address[C. Labourie]{Universit\'e de Lorraine -- CNRS, UMR 7502 IECL, BP 70239
54506 Vandoeuvre-lès-Nancy,  France}
\email{camille.labourie@univ-lorraine.fr}

\address[A. Lemenant]{Universit\'e de Lorraine -- CNRS, UMR 7502 IECL, BP 70239
54506 Vandoeuvre-lès-Nancy,  France}
\email{antoine.lemenant@univ-lorraine.fr}

\date{\today}

\title{Finite number of traces for Mumford-Shah minimizers in dimension 2}

\begin{document}

\begin{abstract}
In this short note, we answer a question raised  by E. De Giorgi, showing that a Mumford-Shah minimizer in dimension 2 can admit at most  three maximum limit values as approaching the singular set. This result stems from tools developed in the early 2000's by G. David, A. Bonnet, and J.-C. L\'eger.     
\end{abstract}

\maketitle


\tableofcontents

\section{Introduction}


Let $\Omega \subset \R^N$ be open and  $g\in L^\infty(\Omega)$. We say that $(u,K)$ is a Mumford-Shah minimizer when it is a solution for the problem
\begin{eqnarray}
\min_{(u,K)} \int_{\Omega \setminus K} |\nabla u|^2 \;dx + \int_{\Omega \setminus K} |u-g|^2 \;dx + \mathcal{H}^{N-1}(K), \label{MSprob}
\end{eqnarray}
where the minimum is taken over all $K\subset \Omega$  relatively closed,  and $u \in H^1(\Omega \setminus K)$.

 The Mumford-Shah problem is a cornerstone in the study of free discontinuity problems within the calculus of variations, garnering significant attention since the 1990s.
A central aspect of this research is the Mumford-Shah conjecture, which proposes a precise characterization of the structure of minimizers, predicting that the singular set is a (locally) finite union of $C^1$ arcs that are almost-disjoint, and whose extremities can be either free (crack-tips) or triple junctions.
We refer to the books \cite{afp,d,delF} for an introduction to the Mumford-Shah functional.  
  
In his famous 1991 paper \cite{dg91}, E. De Giorgi has stated a list of 9 conjectures about free discontinuity problems. After 34 years,  most of them are still unsolved. Actually, to the best of the authors' knowledge, only Conjecture number 1 and 3 are known to be true. Conjecture~1 deals with the   higher integrability of the gradient,  and was established by C. De Lellis and M. Focardi  \cite{delF1}  in dimension~2 and then by G. De Philippis and A. Figalli  \cite{DPF} in any dimension.  Conjecture 3 follows from the $C^1$ regularity result of L. Ambrosio, N. Fusco and D. Pallara \cite{afpArt}. Conjecture 6 is a generalization of  the Mumford-Shah conjecture in higher dimensions.

The present note concerns Conjecture number 2. Indeed, among the list of E. De Giorgi's open questions \cite{dg91}, one can find the following one about the number of possible traces for $u$ when approaching the singular set $K$:

\setcounter{conjecture}{1} 
\begin{conjecture}[E. De Giorgi, Conjecture 2 page 56 of \cite{dg91}]\label{conjecture} If $(u,K)$ is a minimizing pair of the functional in  \eqref{MSprob}, and if we set 
$$E(x)=\bigcap \left\{C \;: \; C\subset \R \text{ closed set such that } \lim_{y\to x} {\rm dist}(u(y),C)=0\right\},$$
then, for every $x\in \Omega$, $E(x)$ has a finite cardinality $\alpha(x)\leq N+1$.
\end{conjecture}

Of course outside $K$ the function $u$ is continuous so that  $\alpha(x)=1$. The question is therefore relevant only for points $x\in K$. In other words, the conjecture says that $u$ must admit a finite number of traces up to the singular set. Moreover, if the Mumford-Shah conjecture is true in dimension 2, then this number must be at most $3$, 
which occurs only at a triple point. This explains the conjectured bound  $\alpha(x)\leq N+1$.

In this note, we prove that the conjecture holds in dimension 2. Here is our main result.  




\begin{theorem}\label{main} If $N=2$, then Conjecture~\ref{conjecture} is true.
\end{theorem}

Actually, the ingredients we use to  prove   Theorem \ref{main} are mainly drawn from the works of G. David, particularly his book \cite{d}, his paper with A. Bonnet \cite{DavidBonnet}, and another paper with J.C. L\'eger \cite{dleger}. However, the full argument requires assembling several statements, together with a few additional  ones, so that a specific and complete proof is actually not directly available in the literature. The aim of these notes is to fill the gaps and provide the necessary details.

Let us explain the scheme of proof for Theorem \ref{main}, along with the plan of the present notes.
In \cite{DavidBonnet}, it is proven that for a global Mumford-Shah minimizer $(u,K)$ in the plane, the connected components of $\R^2 \setminus K$ are John domains with center at infinity. A local version of this result is also available in G. David's book \cite{d} (see Proposition \ref{prop1} below), which provides the existence of ``escaping paths'' in $\Omega \setminus K$ for a Mumford-Shah minimizer. 
This notion, sometimes called local-John condition, is akin to, but weaker than, saying that the components of $\Omega \setminus K$ are John domains.

 
The existence of escaping paths has a key implication for Mumford-Shah minimizers : $u$ satisfies Hölder estimates up to the singular set $K$. Combined with the local finiteness of components (see Proposition \ref{prop2}), this already proves that $\alpha(x)$ is bounded by a universal constant.
But using another powerful result from G. David and J.-C. L\'eger \cite{dleger}, we can go further and actually conclude that $\alpha(x)\leq 3$. Indeed, assuming for contradiction that  $B(x_0,r)\setminus K$ contains more than 3 connected components, one can proceed to a blow up limit and thanks to the local John-domain condition, the number of connected components is preserved in the limit (see Proposition \ref{prop3}). The main result in \cite{dleger} then directly implies that at most 3 components are allowed, and so follows the conclusion. This achieves the plan of proof for   Theorem \ref{main}.

As a byproduct of the proof, we will show the following interesting statement (which is   given by Proposition \ref{prop3}).

\begin{theorem}\label{main1}
    Let $N=2$ and let $(u,K)$ be a Mumford-Shah minimizer in $\Omega$.
    Then for all $x_0\in K$ and $r > 0$ such that $B(x_0,r) \subset \Omega$, there are at most $3$ connected components of $B(x_0,r) \setminus K$ whose closure contains $x_0$.
\end{theorem}


To finish the introduction, let us say a few words about open questions. The conjecture in higher dimensions $N\geq 3$ remains widely open. The first issue is that the existence of escaping paths relies on the local finiteness of the components of $\Omega \setminus K$, which is unknown in dimension higher than $2$.
It is actually an open question asked by G. David  at the end of page 368 of his book \cite{d}. Then  David-Léger's result (Theorem \ref{davLe} below)  relies on a monotonicity formula that has, up to now, no analogue in higher dimensions. All of these missing ingredients lead us to believe that the higher-dimensional conjecture seems out of reach for the moment, even for the particular case $N=3$ for which the list of minimal cones is known.
 
\vspace{0.5cm}

{\bf Acknowledgements.}  The main result of this paper was announced  by the authors during  the small one-day-workshop  in the honor of Guy David for his retirement, which took place at the university of Orsay  on 5th november 2024.  The authors  wish to warmly thank Laurent Moonens  for the invitation and for the organization of this nice event. The first intention of the authors was  to  credit Theorem   \ref{main} to Guy David's name, but he answered the following: this is not necessary since    his paper \cite{DavidBonnet}, that is one of the main ingredient to prove Theorem   \ref{main},   already had the aim to solve a conjecture by E. De Giorgi !

\section{Escaping paths and finite number of components}


One of the main ingredient of this section is the existence of escaping paths that was proved  in  \cite[Proposition 16 page 473]{d}. 
This property means that from every point starts an arclength-parametrized path which gets away linearly from $K$.
It is a local version of a similar result that was first proved in \cite{DavidBonnet} for global minimizers.  We refer to \cite{DavidBonnet} for the proof\footnote{To be more precise, in \cite{DavidBonnet} the statement is given for a so called ``topological almost MS-minimiser with gauge $h$''.   Proposition \ref{prop1} follows from applying \cite[Proposition 16 page 473]{d} to a Mumford-Shah minimizer, which is a particular case of topological almost MS-minimiser with gauge $h(r)=2\|g\|^2_{\infty}r$.}.


\begin{proposition}[Proposition 16 page 473 of \cite{d}] \label{prop1} Let $N=2$ and let $(u,K)$ be a Mumford-Shah minimizer. There is  a constant $\tau > 0$ (which only depends on $\norm{g}_{\infty}$) and a universal constant $c_0 >0$ such that the following holds. For all $x\in \Omega \setminus K$ and $t_0\in (0, \min({\rm dist}(x,\partial \Omega),\tau))$, there exists a $1$-Lipschitz mapping $z:[0,t_0]\to \Omega \setminus K$ such that $z(0)=x$ and 
\begin{eqnarray}
    {\rm dist}(z(t),K)\geq c_0 t \ \text{ for } t\in [0,t_0]. \label{lowcurve}
\end{eqnarray}
\end{proposition}

    In the sequel, we restrict the name ``escaping path'' to the paths that satisfy \eqref{lowcurve} with the precise constant $c_0$. In contrast to John domains, a component of $\Omega \setminus K$ may not have a common center where all points are connected to by an escaping path. From a same point $x \in \Omega \setminus K$ can starts several escaping paths which connects $x$ to different centers further away from $K$, but there may be no good quantitative connection (satisfying (\ref{lowcurve}) with a constant independent of $x$ and $K$) between one center and another.


A key ingredient in the proof of Proposition \ref{prop1} is the local finiteness of the components of $\Omega \setminus K$. We recall the precise statement below. The argument  appears in the proof of \cite[Lemma 13 page 472]{d}, but it applies only in the planar setting, which is the reason why Proposition \ref{prop1} is also restricted to the dimension 2.



\begin{proposition}\label{prop2}
    Let $N=2$, let  $(u,K)$ be a Mumford-Shah minimizer in $\Omega$. There exists a radius $r_0 > 0$ (which only depends on $\norm{g}_{\infty}$) and a universal constant $C \geq 1$ such that for all $x_0 \in K$, and $r \leq r_0$ with $B(x_0,2r) \subset \Omega$, there are at most $C$ components of $B(x_0,2r) \setminus K$ which meet $B(x_0,r)$.
\end{proposition}

\begin{proof} We justify Proposition \ref{prop2} by application of \cite[Lemma 13 page 472]{d}.
    Let $x\in B(x_0,r)\setminus K$ and let $\Omega_0$ be the connected component of $\Omega \setminus K$ that contains $x$. Then a direct application of \cite[Lemma 13 page 472]{d}  says that there exists a universal constant $C_1>0$ such that  $\Omega_0\cap B(x,r/2)$ contains a disk of radius $r/C_1$.
    We deduce that  every component of $B(x_0,2r) \setminus K$ which meets $B(x_0,r)$ has a volume at least $\pi r^2 C_1^{-2}$. 
    Since they are disjoint, there can only be a finite number of such components, bounded by a constant which depends only on $C_1$.

 An alternative argument, making our paper perhaps more self-contained, would be to remark that Proposition \ref{prop1} is actually a stronger statement than the local finiteness. Indeed, by the existence of escaping paths (i.e., Proposition \ref{prop1}), we know that every $x\in B(x_0,r)\setminus K$ is connected to a point $z$ in $B(x_0,2r)$ satisfying ${\rm dist}(z,K)\geq c_0 r$, and then one can conclude as before.
 \end{proof}

\section{At most 3 components in the complement of $K$ locally}

 
We continue our analysis of connected components, by showing that every point $x_0 \in K$ is adherent to at most 3 components.
For that purpose we use a deep result by David and L\'eger contained in \cite{dleger}. 

\begin{theorem}[David-L\'eger \cite{dleger}] \label{davLe} Let $(u,K)$ be a Mumford-Shah global minimizer in the plane such  that $\R^2\setminus K$ is not connected. Then $K$ is a line or a propeller (i.e. three half lines meeting at one point with equal angles), and $u$ is locally constant on $\R^2\setminus K$.
\end{theorem}

We prove the following proposition that has been already announced in the introduction (Theorem \ref{main1}), that we restate it here again.

\begin{proposition}\label{prop3}   Let $N=2$ and let  $(u,K)$ be a Mumford-Shah minimizer in $\Omega$.
    Then for all $x_0\in K$ and $r > 0$ such that $B(x_0,r) \subset \Omega$, there are at most $3$ connected components of $B(x_0,r) \setminus K$ whose closure contains $x_0$.
\end{proposition}

\begin{proof}  We can assume that $x_0$ is the origin.
    We let   $c_0>0$ be the constant of Proposition~\ref{prop1}  and $r_0 > 0$ be a radius such that $B(x_0,r_0) \subset \Omega$.
    Let us denote by $\{A_n\}$  the connected components of $B(0,r_0) \setminus K$ whose boundaries contain $0$ and assume by contradiction that  $\sharp \{A_n\}\geq 4$. Consider a blow-up sequence $(u_j, K_j)$ of $(u,K)$ at the origin, namely $K_j=r_j^{-1}K$ and $u_j(x)= r_j^{-1/2}u(r_j x)$ for a sequence $r_j\to 0$. This blow-up sequence converges, up to a subsequence still denoted by $r_j\to 0$, to a global minimizer $(u_{\infty}, K_{\infty})$ (see \cite[Section 40]{d} for the notion of blow-up limit). Our goal is to prove that $\R^2 \setminus K_{\infty}$ contains at least $4$ different connected components, which would be a contradiction with Theorem \ref{davLe}. 


    Since $0\in \partial A_n$ for all $n$, we know that for all $j$ there exists $x_n \in A_n \cap B(0,r_j)$. By applying Proposition \ref{prop1}, we deduce that this point $x_n$  is connected to a point $y_n$ such that $|x_n-y_n|\leq r_j$ and ${\rm dist}(y_n, K)\geq c_0r_j$. Up to a subsequence, the points $y_n/r_j \in B(0,2)$  converges as $j\to +\infty$ to some points $z_n$ in $\overline{B}(0,2)$. Since ${\rm dist}(y_n, K)\geq c_0r_j$ for all $j$, we deduce that actually $z_n \in \overline{B}(0,2)\setminus K_{\infty}$ for all $n$ and that even ${\rm dist}(z_n, K_{\infty})\geq c_0$ for all $n$. Let us show that: 
\begin{eqnarray}
    \text{the points $\{z_n\}_n$ all belong to different connected components of $\R^2\setminus K_{\infty}$. } \label{amontrer}
\end{eqnarray}

    Assume by contradiction that there is a curve $\gamma$ from $z_n$ to $z_{n'}$ in $\R^2 \setminus K_{\infty}$ (notice that  connected open set is always arc-wise connected). Then by local Hausdorff convergence, $\gamma$ does not touch  any of the sets $K_j$ for $j$ large enough. Also,  by convergence of $y_n/r_j$ to $z_n$, we know that for $j$ large enough   $|y_n/r_j-z_n|\leq c_0/2$. Since moreover ${\rm dist}(z_n, K_{\infty})\geq c_0$, we know that for $j$ large enough, the segments  $[y_n/r_j, z_n] \subset \overline{B}(z_n,c_0/2)$ does not touch $K_{\infty}$ and again by Hausdorff convergence, does not touch $K_j$ neither. Therefore we can modify the curve $\gamma$ as follows: 
$$\gamma'=\gamma \cup [y_n/r_j,z_n] \cup [y_{n'}/r_j,z_{n'}].$$
    This provides a continuous curve from $y_n/r_j$ to $y_{n'}/r_j$ which does not touch $K_j$. In other words the curve $r_j \gamma'$ is a continuous curve from $y_n$ to $y_{n'}$, that does not touch $K$. Since $\gamma'$ is bounded the curve $r_j\gamma'$ belongs to  $B(0,r_0)\setminus K$ for $j$ large enough, which is  a contradiction with the fact that $y_n$ and $y_{n'}$ are in different connected components of  $B(0,r_0)\setminus K$. This achieves the proof of the claim in \eqref{amontrer}.

But then we have shown that $K_{\infty}$ is the singular set associated to a global minimizer in the plane for which $\R^2\setminus K_{\infty}$ has more than $4$ connected components. This is not possible due to Theorem \ref{davLe}, and so follows the Proposition. 
\end{proof}

\begin{remark}
 A related blow-up argument can be found in the recent book \cite{delF}, namely in the proof of Theorem 1.5.2 (ii) (see page 102-103 of \cite{delF}, proof of Corollary 4.4.2). There, the stability of connected components is instead established through the use of $C^1$ estimates and $\varepsilon$-regularity theory.
\end{remark}

\begin{remark} With the same proof as for Proposition \ref{prop3} we could also have proved the following alternative ``global'' version of the statement: if $(u,K)$ is a Mumford-Shah minimizer in $\Omega$ and $x_0\in K$, then the number of connected components of $\Omega \setminus K$ whose closure contain $x_0$, does not exceed 3. 
\end{remark}

 \section{Gradient bound   and H\"older estimates}

Finally, we prove the existence of a trace, up to the boundary, thanks to the existence of escaping paths in the components.   For that purpose we will need an $L^\infty$ estimate on the gradient of $u$. We start with a very well known upper bound.

\begin{proposition}\label{prop333}Let $N=2$ and let $(u,K)$ be a Mumford-Shah minimizer  {in $\Omega$}. Then for all $x\in \Omega$ and $r>0$ such that $B(x,r)\subset \Omega$ we have 
\begin{eqnarray}
    \int_{B(x,r)} |\nabla u|^2 \; dx + \HH^1(K\cap B(x,r)) \leq 2\pi r + 2\pi \|g\|_{\infty}r^2. \label{upperbound}
\end{eqnarray}
\end{proposition}
\begin{proof} This is the very standard upper bound on Mumford-Shah  minimizers. First notice that $\|u\|_\infty \leq \|g\|_\infty$ because one can compare $u$ with a truncated version of $u$ (see for instance \cite[Lemma 2.1.1]{delF}).  Assume without loss of generality that $\overline{B}(x,r) \subset \Omega$ and use the competitor $(v,L)$ made by taking $L=K\cup \partial B(x,r) \setminus B(x,r)$ and $v=0$ in $B(x,r)$, $v=u$ in $\Omega \setminus B(x,r)$. This automatically gives \eqref{upperbound}.
\end{proof}

In the sequel we will also need the following gradient estimate on the function $u$. This is the only place where we use that $(u,K)$ is a Mumford-Shah minimizer  (and not only an almost-minimizer).

 \begin{proposition} \label{prop33} 
     Let $N=2$ and let $(u,K)$ be a Mumford-Shah minimizer in $\Omega$. There exists a constant $r_2 > 0$ (which only depends on $\norm{g}_{\infty}$) such that for all $x_0 \in K$, for all $r \leq r_2$ with $B(x_0,2r) \subset \Omega$ and for all $x \in B(x_0,r) \setminus K$, we have
 \begin{eqnarray} 
    |\nabla u|(x) \leq 2 {\rm dist}(x,K)^{-\frac{1}{2}} + 4 (\|g\|_\infty {\rm dist}(x,K))^{\frac{1}{2}}. \label{estimateH}
\end{eqnarray}
\end{proposition}
 
\begin{proof} The proposition follows from the fact that  $u$ is the solution of the equation 
$$-\Delta u = u-g  \text{  in } \Omega \setminus K$$
 with $u-g \in L^\infty$. In particular, $u$ is of class $C^{1,\alpha}$ by elliptic regularity.
 
  In order to derive some more precise estimates it will be convenient to  assume  for a moment that  $g$ is also of class $C^1$. In this case we can  use the equation to get
 $$-\Delta (\partial_i u) = \partial_i (u-g).$$
Let $x\in \Omega \setminus K$. To get the desired gradient estimate we will use a variant of the mean value property, applied to $\partial_i u$.

Assuming without loosing generality  that $x=0$, we use  a standard computation in polar coordinates (see \cite[page 26]{evans}) which yields, for all $r > 0$ such that $\overline{B}(0,r) \subset \Omega \setminus K$,
 \begin{eqnarray}
 \frac{d}{dr} \left( \frac{1}{\pi r^2}\int_{B(0,r)} \partial_i u \;dy\right)&= &\frac{1}{2\pi r} \int_{B(0,r)} \Delta (\partial_i u)\;dy \notag \\
 &=& \frac{1}{2\pi r} \int_{B(0,r)} \partial_i (g-u)\;dy \notag\\
 &=&  \frac{1}{2\pi r} \int_{\partial B(0,r)}  (g-u)  \frac{y_i}{r}\;dy. \notag
 \end{eqnarray}

 We deduce that for all $s<r$,
 $$ \frac{1}{\pi r^2}\int_{B(0,r)} \nabla u \;dy - \frac{1}{\pi s^2}\int_{B(0,s)} \nabla u \;dy= \int_{s}^r \left( \frac{1}{2\pi t} \int_{\partial B(0,t)}  (g-u) \frac{y}{t} \;dy \right)\,dt.$$
Since $\nabla u$ is known to be continuous we can pass to the limit $s\to 0$ to get the representation formula
 $$\nabla u(0)= \frac{1}{\pi r^2}\int_{B(0,r)} \nabla u \;dy+\int_{0}^r \frac{1}{2\pi t} \int_{\partial B(0,t)}  (u-g)  \frac{y}{t}\;dy \,dt,$$
 from which we deduce that 
 \begin{eqnarray}
 |\nabla u(0)|\leq  \frac{1}{\pi r^2}\int_{B(0,r)} |\nabla u| \;dy + \|u-g\|_{\infty} r. \label{estimateMate}
 \end{eqnarray}
  Then if $g$ is not $C^1$ we can regularize by a mollifier $\varphi_\varepsilon$ and denoting by $u_\varepsilon:=u * \varphi_\varepsilon$ and $g_\varepsilon:= g *\varphi_\varepsilon$  we observe that 
$$-\Delta u_\varepsilon= u_\varepsilon-g_\varepsilon.$$
 and $\norm{u_{\varepsilon} - g_{\varepsilon}}_{\infty} \leq \norm{u - g}_{\infty}$.
 Applying the above computations to $u_\varepsilon$ we get
 $$|\nabla u_\varepsilon(0)|\leq  \frac{1}{\pi r^2}\int_{B(0,r)} |\nabla u_\varepsilon | \;dy + \|u -g \|_{\infty}r,$$
 and letting  $\varepsilon\to 0$ we recover \eqref{estimateMate} in the general case when $g \in L^\infty$.  Finally, using that $\|u\|_\infty\leq \|g\|_\infty$ and applying on H\"older inequality in \eqref{estimateMate} we arrive to 
\begin{eqnarray}
    |\nabla u (0)|& \leq &  \left(\frac{1}{\pi r^2} \int_{B(0,r)} |\nabla u  |^2 \;dy\right)^{\frac{1}{2}}+ \|g \|_{\infty}r^2 \notag \\
&\leq & \left(\frac{2}{r}+ 2\|g\|_\infty r\right)^{1/2}+ 2\|g \|_{\infty}r \notag \\
&\leq &  \left(\frac{2}{r}\right)^{1/2}+\left(2 \|g\|_\infty r\right)^{1/2}+ 2\|g \|_{\infty}r \notag \\
&\leq &  2r^{-1/2} + 4(r\|g\|_{\infty})^{1/2},
\end{eqnarray}
where we have used the standard energy bound $\int_{B(x,r)}|\nabla u|^2\; dx \leq 2\pi r+2\pi \|g\|_\infty r^2$ given by Proposition \ref{prop333}, and the fact that  $2r\|g\|_\infty\leq 1$. Now, remember that these computations require $\overline{B}(z,r) \subset \Omega \setminus K$. One can let $r \to \mathrm{dist}(z,K)$ for instance if $z$ lies in a ball $B(x_0,R)$ where $x_0 \in K$, $R \leq 1/(2\norm{g}_{\infty})$ and $B(x_0,2R) \subset \Omega$.
\end{proof}

Next, by arguing similarly to Lemma 21.3  of \cite{DavidBonnet}, we get the following proposition. The difference with  \cite{DavidBonnet}  is that here we state it for  minimizers of the Mumford-Shah functional  in a domain $\Omega$, while  \cite{DavidBonnet} focuses on  global minimizers only.
 
More precisely, once the existence of escaping path (Proposition \ref{prop1}) and the gradient bound (Proposition \ref{prop33}) are given,  the proof can be obtained exactly as in  Lemma 21.3  of \cite{DavidBonnet}. Since the proof is not so long, and crucial for our main result, we  provide the full details for the reader's convenience.

Let us mention that  it is   not difficult  to derive an H\"older estimate  along a single escaping  path for a function that satisfies a gradient estimate such as \eqref{estimateH}.  This is actually ``standard'' in the theory of John domains. 
But the main issue is that  we do not have directly at disposal a nice  curve connecting two given points $x$ and $y$ passing through a common center as it would happen in a John domain. We only know the existence of escaping curves going far away from $K$,  ending at multiple possible centers.
This explains the delicate and  very original proof that was written by G.~David in \cite{DavidBonnet}, and that we write again here for sake of completeness.

\begin{proposition}[\cite{DavidBonnet}] \label{prop3bis} Let $N = 2$ and $(u,K)$ be a Mumford-Shah minimizer in $\Omega$. There exists a constant $r_3>0$ (which only depends on $\norm{g}_{\infty}$) and a universal constant $C \geq 1$ such that for all $r \leq r_3$ with $B(x_0,4r) \subset \Omega$ and for all $x,y$ that lie in the same connected component of $B(x_0,r) \setminus K$ we have
$$|u(x)-u(y)|\leq C r^{1/2}.$$
\end{proposition}

\begin{proof}  We can assume that $x_0$ is the origin. We let $c_0,\tau>0$ be the constants of Proposition~\ref{prop1}. We may assume that $r_3\leq \tau$ and we let $r \leq r_3$ be given such that $B(0,4r) \subset \Omega$. Recall that Proposition~\ref{prop1} guarantees the existence of escaping paths, which means in particular that for all $x\in B(0,r) \setminus K$, there exists a $1$-Lipschitz mapping $z:[0,r_3]\to \Omega \setminus K$ such that $z(0)=x$ and 
\begin{eqnarray}
    {\rm dist}(z(t),K)\geq c_0 t \text{ for } t\in [0,r]. \label{lowcurve0}
\end{eqnarray}
 We can also assume $r_3 \leq r_2/2$ so that Proposition \ref{prop33} also applies and says in particular that for all $x\in B(0,2r) \setminus K$,
 \begin{eqnarray} 
    |\nabla u|(x)\leq 2 {\rm dist}(x,K)^{-\frac{1}{2}} + C_1, \label{estIM}
\end{eqnarray}
with $C_1= 4 (2\|g\|_\infty  r_3)^{1/2}$.

Now let $x,y$ be two points in the same connected component of $B(0,r)\setminus K$, as in the statement of the proposition.
 Set $B=B(0,r)$ and $B_1:=B(0,2r)$.
Then, choose a finite subset $E$ of $B_1\setminus K$ which is $10^{-1}c_0r$ dense in $B_1\setminus K$. We can find such a set $E$ with less than $N_0$ elements, depending only on $c_0$. 

    Let $\xi:[0,1]\to B\setminus K$ be a continuous curve such that  $\xi(0)=x$ and $\xi(1)=y$.  The rest of the proof will be  carried out in several steps, following the lines of   \cite[Lemma 21.3]{DavidBonnet}. 

\medskip

\emph{Step 1. Definition of $E(t)$ and $z(t)$}.  For $t\in[0,1]$ we define $E(t)\subset E$ as being the set of points $z\in E$ which can be connected to $\xi(t)$ by an escaping path $\gamma : [0,r] \to B(0,2r) \setminus K$ that is $2$-Lipschitz, satisfies  $\gamma(0)=\xi(t)$, $\gamma(r)=z$, and such that 
\begin{eqnarray}
{\rm dist}(\gamma(s),K)\geq c_0 10^{-1} s \text{ for } 0\leq s\leq r. \label{escapeFin}
\end{eqnarray} 
Notice that we have added a factor $10^{-1}$ in \eqref{escapeFin} compared to \eqref{lowcurve0}, which a less restrictive condition. It is easy to see that $E(t)$ is not empty. Indeed, by \eqref{lowcurve0} we have  the existence of an escaping path $\gamma : [0,r] \to B(0,2r) \setminus K$ starting from $\xi(t)$, which is $1$-Lipschitz, and  satisfies something 10 times more restrictive than \eqref{escapeFin}. The   path fulfills almost all the requirement on $\gamma$ excepted that its endpoint $z'$ may not belong to the finite set $E$. But there exists $z\in E$ such that $|z-z'|\leq c_010^{-1}r$.
And we also know that  ${\rm dist}(z',K)\geq c_0r$.
Therefore, adding the segment $|z-z'| \subset B(0,2r) \setminus K$ to the path $\gamma$ and re-parameterizing it on $[0,r]$ as a $2$-Lipschitz function, yields the desired path, which proves that $E$ is not empty. 

Since there is some room in the inequalities, it is also easy to prove the following fact that will be used later: for all $t\in [0,1]$, we can find a point $z(t)\in E(t)$ which also satisfies
\begin{eqnarray}
z(t) \in E(t') \text{ for all } t' \text{ in some neighborhood of } t \in [0,1]. \label{subtile}
\end{eqnarray}
 This is possible since for all $t'$ close enough to $t$, the point $\xi(t')$ is connected to $\xi(t)$ by a small segment in $B(0,2r) \setminus K$.
 
 \medskip
 
\emph{Step 2. H\"older estimate for points in $E(t)$}. The next step is to prove that 
\begin{eqnarray}
|u(\xi(t))-u(z))|\leq C r^{1/2} \text{ for all } z \in E(t). \label{holder1}
\end{eqnarray}
For this purpose we take a path $\gamma$ from $\xi(t)$ to $z$ satisfying \eqref{escapeFin} and we use the gradient estimate  \eqref{estIM}.  Notice that $u$ is a $C^1$ function in $\Omega \setminus K$. We can then write
\begin{eqnarray}
|u(\xi(t))-u(z)|&\leq & \int_{0}^r |\nabla u(\gamma(s)| |\gamma'(s)| \,ds\notag \\
&\leq & 2 \int_{0}^r   2 {\rm dist}(\gamma(s),K)^{-\frac{1}{2}} + C_1 \,ds\notag \\
&\leq & 2 \int_{0}^r   2 c_0^{1/2}10^{-1/2} s^{-\frac{1}{2}} + C_1 \,ds\notag \\
&\leq& C r^{1/2}, \notag
\end{eqnarray}
where $C$ depends only on $c_0$ and $C_1$, which proves  \eqref{holder1}.

\medskip

\emph{Step 3. Definition of $t(z)$ and iterative argument}. Next we define
$$E_1:=\{z\in E \; |\; z\in E(t) \text{ for some } t \in [0,1]\},$$
and then for $z\in E_1$ we define
$$t(z):=\sup \{t \; |\; z \in E(t)\}.$$
Notice that for $z\in E_1$, passing to the supremum and using the continuity of $u\circ \xi$, we still have from \emph{Step 2} the following estimate
\begin{eqnarray}
|u(\xi(t(z)))-u(z)|\leq Cr^{1/2}, \label{estimSup}
\end{eqnarray}
with the same constant $C>0$ as before.

We start now an iterative argument: set $t_0=0$ and $z_0=z(t_0)$, as defined in \eqref{subtile}. Then $z_0\in E(t_0)$ and by \eqref{holder1} we get
$$|u(x)-u(z_0)|\leq Cr^{1/2}.$$
Next, set $t_1=t(z_0)$. It follows that $t_1>t_0$, by definition of $z_0=z(t_0)$ and because of \eqref{subtile}. We also  know from \eqref{estimSup} that 
\begin{eqnarray}
|u(\xi(t_1))-u(z_0)|\leq Cr^{1/2}. \label{estimSup}
\end{eqnarray}
We continue like this iteratively, by defining $z_i=z(t_i)$ and $t_i=t(z_{i-1})$, until we reach $t_i=1$. The main point is that $t_{i+1}>t_i$ for all $i$ (as already mentioned above for $t_0$ and $t_1$), and more important:  when $t_i < 1$,
\begin{eqnarray}
\text{each $z_i$ is different from all its predecessors.} \label{aMontrer} 
\end{eqnarray}
Indeed,  $z_i=z(t_i)$, hence $z_i \in E(t')$ for values $t'>t_i$. But then for all $j<i$, $t_i\geq t_{j+1}=t(z_j)$ thus $z_j$ could not lie in $E(t')$ for any $t'>t_i$, as $z_i$ does. So $z_i \not = z_j$, as needed. This proves \eqref{aMontrer}.
Now we remember that $E$ was a finite set, with cardinality $N_0$ bounded by a constant depending only on the constant $c_0$. Since all the $z_i$ are disjoint, the iterative argument must stop after a maximum of $N_0$ steps. This means that $t_i$ reach the maximum value of $1$ after possibly $N_0$ iterations or less. 

\medskip

\emph{Step 4. Conclusion}. We then conclude by adding a maximum of $N_0$ inequalities of the following two  types: the first one is 
\begin{eqnarray}
|u(\xi(t_{i}))-f(z_i)|\leq Cr^{1/2},
\end{eqnarray}
valid because $z_i=z(t_i)\in E(t_i)$ thus \eqref{holder1} applies, and the second one is
\begin{eqnarray}
|u(\xi(t_{i+1}))-f(z_i)|\leq Cr^{1/2},
\end{eqnarray}
valid because $\xi(t_{i+1})=\xi(t(z_i))$ and \eqref{estimSup} applies.  With the triangle inequality, we end up to $\abs{u(x) - u(y)} \leq 2N_0C r^{1/2}$.
\end{proof}

 \section{Proof of main result}
 \label{proofMain}

We are now in position to prove the main result of this note.

\begin{proof}[Proof of Theorem \ref{main}]
    Let $x_0\in K$ be given.  Notice that  the set $E(x_0)$ defined in Conjecture~\ref{conjecture}, is exactly the set of values $\ell \in \R$ for which there exists a sequence $(y_k)_k$ in $\Omega \setminus K$ such that $y_k \to x_0$ and $u(y_k) \to \ell$. 

Assume for a contradiction that $E(x_0)$ contains more than $3$ points. Let $a_k$, $k=1,\dots , 4$ be four points in $E(x_0)$ and set $\delta:= \min_{k,k'}|a_k-a_{k'}|$, the minimal distance between them.
From Proposition \ref{prop3}, we can find a radius  $r_0 > 0$,  such that $B(x_0,r_0) \subset \Omega$ and $B(x_0,r_0)\setminus K$ has at most $C_0$ connected components intersecting $B(x_0,r_0/2)$. 
We can additionnally choose $r_0$ so small that
\begin{eqnarray}
2Cr_0^{1/2}\leq \delta/10, \label{picc}
\end{eqnarray}
where $C>0$ is the constant from Proposition \ref{prop3bis}. 
    Since $x_0$ is adherent to only finitely many components of $B(x_0,r_0)\setminus K$, we can take a smaller radius $r_1<r_0/2$ such that all points of $B(x_0,r_1)\setminus K$ belong to a component of $B(x_0,r_0)\setminus K$ touching $x_0$. Moreover, Proposition \ref{prop3} ensures that there are at most $3$ such components, which we denote by $\{A_n\}$ with $n=1,2,3$, and $A_n$ could possibly be empty.

    Now let $y_k\to x_0$ be a sequence converging to $x_0$. Since $u$ is bounded, there is a subsequence (not relabelled) such that $u(y_k)$ converges to a value $\ell \in \R$.
  We can further extract a subsequence such that the sequence $(y_k)_k$  lies in a single component $A_n$. 
Then, consider another sequence $(z_k)_{k}$ converging to $x_0$ in the same component $A_n$, with $u(z_k)\to \ell'$. By Proposition \ref{prop3bis}, it holds
$$|u(y_k)-u(z_{k})|\leq C r_0^{1/2}$$
and passing to the limit, we obtain
$$|\ell-\ell'|\leq C r_0^{1/2}.$$
    We deduce that  all the possible limit values $\ell$ coming from $A_n$ are contained in an interval $I_n\subset \R$ of length at most $2Cr_0^{1/2}$. 
 From \eqref{picc} we  get   $|I_n|\leq \delta /10$  (and   $I_n$ could be empty if $A_n$ is empty). Then we have  
  $$E(x_0)\subset I_1 \cup I_2 \cup I_3.$$
 But this contradicts the fact that $E(x_0)$ contains four distinct points of mutual distance $|a_k-a_{k'}|\geq \delta$. We therefore conclude that $E(x_0)$ contains at most three points.
\end{proof}


\bibliographystyle{plain}
\bibliography{biblioDG}

\end{document}